\documentclass[a4paper,12pt]{article}
\usepackage{amsmath}
\usepackage{amssymb}
\usepackage{amscd}
\usepackage{amsthm}
\usepackage{graphicx}
\usepackage{setspace}

\newtheorem{theorem}{Theorem}
\newtheorem{lemma}[theorem]{Lemma}
\newtheorem*{rtheorem}{Main Theorem}
\newtheorem{corollary}[theorem]{Corollary}

\newtheorem{remark}[theorem]{Remark}
\usepackage{pdfsync}
\def \PG{\mathrm{PG}}

\def\V{\mathrm{V}}

\def\C{\mathcal{C}}

\def\D{\mathcal{D}}

\def\F{\mathbb{F}}
\def\G{\mathcal{G}}
\def\I{\mathcal{I}}

\def\P{\mathcal{P}}
\def\M{\mathcal{M}}
\def\S{\mathcal{S}}
\def\O{\mathcal{O}}
\def\L{\mathcal{L}}
\def\K{\mathcal{K}}

\def\PGammaL{\mathrm{P}\Gamma\mathrm{L}}

\title{Extending pseudo-arcs in odd characteristic}
\author{Tim Penttila \and Geertrui Van de Voorde}
\begin{document}
\maketitle
Address and affiliation of the authors:

\noindent
Tim Penttila\\
Department of Mathematics\\
Colorado State University\\
Fort Collins, CO 80523-1874\\
penttila@math.colostate.edu\\

\noindent
Geertrui Van de Voorde\\
Department of Mathematics\\
Vrije Universiteit Brussel\\
Pleinlaan 2\\
1050 Brussel\\
gvdevoor@vub.ac.be\\ 0032/26293476

\begin{abstract} A {\em pseudo-arc} in $\PG(3n-1,q)$ is a set of $(n-1)$-spaces such that any three of them span the whole space. A pseudo-arc of size $q^n+1$ is a {\em pseudo-oval}. If a pseudo-oval $\O$ is obtained by applying field reduction to a conic in $\PG(2,q^n)$, then $\O$ is called a {\em pseudo-conic}. 

We first explain the connection of (pseudo-)arcs with Laguerre planes, orthogonal arrays and generalised quadrangles. In particular, we prove that the Ahrens-Szekeres GQ is obtained from a $q$-arc in $\PG(2,q)$ and we extend this construction to that of a GQ of order $(q^n-1,q^n+1)$ from a pseudo-arc of $\PG(3n-1,q)$ of size $q^n$.

The main theorem of this paper shows that if $\K$ is a pseudo-arc in $\PG(3n-1,q)$, $q$ odd, of size larger than the size of the second largest complete arc in $\PG(2,q^n)$,  where for one element $K_i$ of $\K$, the partial spread $\S=\{K_1,\ldots,K_{i-1},K_{i+1},\ldots,K_{s}\}/K_i$ extends to a Desarguesian spread of $\PG(2n-1,q)$, then $\K$ is contained in a pseudo-conic. The main result of \cite{Casse} also follows from this theorem.

\end{abstract}

\section{Introduction}
Throughout this paper, we denote the finite field of order $q=p^h$, where $p$ is prime, by $\F_q$ and the $n$-dimensional projective space over $\F_q$ by $\PG(n,q)$. A {\em pseudo-arc} is a set of $(n-1)$-spaces in $\PG(3n-1,q)$, such that any three span the whole space. (In \cite{Casse}, a pseudo-arc of size $k$ is called a {\em $k$-set} in $\PG(3n-1,q)$.) If $n=1$, a pseudo-arc of size $k$ is called a {\em k-arc}. It is shown in \cite{Thas2} that a pseudo-arc in $\PG(3n-1,q)$ has at most $q^n+1$ elements if $q$ is odd, and at most $q^n+2$ elements if $q$ is even. A pseudo-arc of size $q^n+1$ is a {\em pseudo-oval}, a pseudo-arc of size $q^n+2$ is called a {\em pseudo-hyperoval}. If $n=1$, pseudo-ovals and pseudo-hyperovals are simply called {\em ovals} and {\em hyperovals}. In \cite{Thas2}, Thas showed that if $q$ is even, a pseudo-oval of $\PG(3n-1,q)$ can be completed (in a unique way) to a pseudo-hyperoval. 

An example of a pseudo-oval can be obtained by {\em field reduction}: to every point of $\PG(2,q^n)$, there corresponds in a natural way an $(n-1)$-space of $\PG(3n-1,q)$. When we apply field reduction to an oval of $\PG(2,q^n)$, we obtain a pseudo-oval, and a pseudo-oval that is obtained in this way from a conic in $\PG(2,q^n)$ is called a {\em pseudo-conic}. 

In this paper, we prove a condition under which a pseudo-arc in $\PG(3n-1,q)$, $q$ odd, can be completed to a pseudo-conic. More precisely, we will show the following.
\begin{rtheorem}
If $\K=\{K_1,\ldots,K_s\}$ is a pseudo-arc in $\PG(3n-1,q)$, $q$ odd, of size at least the size of the second largest complete arc in $\PG(2,q^n)$,  where for one element $K_i$ of $\K$, the partial spread $\S=\{K_1,\ldots,K_{i-1},K_{i+1},\ldots,K_{s}\}/K_i$ extends to a Desarguesian spread of $\PG(2n-1,q)=\PG(3n-1,q)/K_i$, then $\K$ is contained in a pseudo-conic.
\end{rtheorem}

In the planar case, the answer to this question depends on the plane being Desarguesian or not and on the size of the arc (see Section 2). Thus the bound cannot be improved without improving the upper bound
on the size of the second largest complete arc in $\PG(2,q^n)$, and the hypothesis on the partial spread
is necessitated by the existence of large complete arcs in non-Desarguesian planes.

We will explain the link between arcs and generalised quadrangles in Section 3. This gives us an easier description of the Ahrens-Szekeres GQ. We then show that also large pseudo-arcs give generalised quadrangles and discuss the associated Laguerre planes to pseudo-ovals in Section 4. In Section 5, we explain the link between Laguerre planes and orthogonal arrays, and in Section 6, we list some computer results on pseudo-arcs in spaces of small order. Finally, in Section 7, we prove the main theorem.

\section{Large arcs in projective planes}

An arc $A$ is called {\em complete} if there is no larger arc containing $A$. When studying complete arcs in projective planes, the situation turns out to be different for Desarguesian and non-Desarguesian planes, and for planes of even and odd order.

In the case that the plane is Desarguesian, and its order is odd, we have the following well-known theorem of Segre.

\begin{theorem} {\rm \cite{Segre}} In $\PG(2,q)$, $q$ odd, every oval is a conic.
\end{theorem}

In the same year, Segre showed the following.
\begin{theorem}{\rm\cite{Segre2}}\label{incompleteq}
In $\PG(2,q)$, $q$ odd, every $q$-arc can be completed to a conic. 
\end{theorem}
In the case that $q$ is even, Tallini showed that a $q$-arc in $\PG(2,q)$, $q$ even, can be completed to a hyperoval \cite{Tallini}. (Qvist had shown that an oval in any plane of even order completes to a hyperoval in 1952 \cite{qvist}). It is easy to see that the extension of a $q$-arc to an oval is unique if $q$ is odd,  and the extension of a $q$-arc to a hyperoval is unique if $q$ is even and larger than two. As opposed to ovals in $\PG(2,q)$, $q$ odd, hyperovals in $\PG(2,q)$, $q$ even, have not been classified. We refer to \cite{Penttila} for an overview of the known hyperovals.

The incompleteness of $q$-arcs in $\PG(2,q)$ raises the question of determining the size of the {\em second} largest complete arc in $\PG(2,q)$, which we will denote by $m_2'(q)$. The answer to this question remains open. The following results for odd $q$, giving upper bounds for $m_2'(q)$, are currently the best known.

\begin{theorem} Let $m_2'(q)$ denote the size of the second largest complete arc in $\PG(2,q)$, $q=p^h$, $p$ prime, $p>2$, then
\begin{itemize}
\item[(1)] \cite{Hirschfeld1} $m_2'(q)\leq q-\sqrt{q}/2+5$ if $p\geq 5$.
\item[(2)] \cite{Hirschfeld2} $m_2'(q)\leq q-\sqrt{q}/2+3$ if $q\geq 23^2$,  $q\neq5^5, 3^6$, $h$ even for $p = 3$.
\item[(3)] \cite{Voloch2} $m_2'(q)\leq q-\sqrt{pq}/4+\frac{29p}{16}+1$ if $h$ is odd.
\item[(4)]\cite{Voloch} $m_2'(q)\leq \frac{44}{45}q+\frac{8}{9}$ if $h=1$.
\end{itemize}

\end{theorem}

There exist non-Desarguesian projective planes that contain complete $q$-arcs. For odd order, there are examples of complete $q$-arcs in the Hall plane of order $9$ (Menichetti \cite{Men1}) and the Hughes plane of order $9$ (Barlotti \cite{Barlotti}). Complete $q$-arcs in all Hall planes of even order were constructed by Menichetti \cite{Men2}.

\section{Generalised quadrangles from arcs}

The generalised quadrangle $T_2(\mathcal{O})$, where $\mathcal{O}$ is an oval of $\PG(2,q)$, was constructed by Tits and published by Dembowski \cite{Dembowski} as follows. Let $\mathcal{O}$ be an oval of $H=\PG(2,q)$, embed $H$ in $\PG(3,q)$ and define the following incidence structure:

\begin{itemize}
\item[$\P$:] (1) the affine points of $\PG(3,q)$\\
(2) the planes $\neq H$ through a tangent line to $\mathcal{O}$\\
(3) the plane $H$
\item[$\L$:] (1) the lines through the points of $\mathcal{O}$, not in $H$\\
(2) the points of $\mathcal{O}$.
\end{itemize}
Then $(\P,\L,\I)$, where incidence is natural, is a GQ of order $q$. Ahrens and Szekeres \cite{AS} and M. Hall, Jr \cite{Hall}, gave a construction of a GQ of order $(q-1,q+1)$ using a hyperoval at infinity. More precisely, when $\mathcal{H}$ is a hyperoval in $H=\PG(2,q)$, $q$ even, where $H$ is embedded in $\PG(3,q)$, the incidence structure $(\P,\L)$, where 
\begin{itemize}
\item[$\P$:] the affine points of $\PG(3,q)$
\item[$\L$:] the lines through the points of $\mathcal{H}$, not in $H$
\end{itemize}
is a GQ$(q-1,q+1)$.
For $q$ odd, Ahrens and Szekeres gave a rather complicated description of a GQ$(q-1,q+1)$ by using cubic curves in a plane \cite{AS}. We denote this GQ by AS$(q)$.

A {\em regular point} $x$ of a GQ of order $(s,t)$ is a point for which $|\{x,y\}^{\perp \perp}|=t+1$ for all points $y\not \sim x$.

{\em Payne derivation} is a process, starting with a GQ of order $q$, with points $\P$ and lines $\L$ that has a {\em regular} point, say $x$, to obtain a GQ $(\P',\L')$ of order $(q-1,q+1)$ as follows (see \cite{Payne0}):
\begin{itemize}
\item[$\P'$:] $\P\setminus x^\bot$
\item[$\L'$:] (1) elements of $\L$, not containing $x$\\
(2) $\{x,y\}^{\bot\bot}$, where $y$ and $x$ are not collinear in $(\P,\L)$ (i.e. the {\em hyperbolic lines} through $x$).
\end{itemize}

The following construction for a generalised quadrangle using a $q$-arc was given by Payne \cite{Payne2}. Let $\K$ be a $q$-arc in $H=\PG(2,q)$, $q$ even, and let $Q$ and $R$ be the points completing $\K$ to a hyperoval. Embed $H$ as the plane at infinity of $\PG(3,q)$, where $q$ is even. Let $(\P,\L,\I)$ be the following incidence structure, where incidence is natural:

\begin{itemize}
\item[$\P$:] the lines through the points of $\K$, not in $H$
\item[$\L$:] (1) the affine points of $\PG(3,q)$\\
(2) the planes through a line of the form $\langle P,Q\rangle$, $P$ in $\K$\\
(3) the planes through a line of the form $\langle P,R\rangle$, $P$ in $\K$.
\end{itemize}
It is not too hard to check that $(\P,\L,\I)$ is a generalised quadrangle of order $(q-1,q+1)$ (the proof goes along the lines of the proof of Theorem \ref{triviaal}).

 We can modify this construction to construct a generalised quadrangle of order $(q-1,q+1)$. This generalised quadrangle is isomorphic to the Ahrens-Szekeres GQ of order $q$ if $q$ is odd.

\begin{theorem}\label{triviaal}
Let $\K$ be a $q$-arc in $H=\PG(2,q)$ and embed $\PG(2,q)$ in $\PG(3,q)$. Let $(\P,\L,\I)$ be the following incidence structure, where incidence is natural.

\begin{itemize}
\item[$\P$:] the lines through the points of $\K$, not in $H$
\item[$\L$:] (1) the affine points of $\PG(3,q)$\\
(2) the planes, different from $H$, through a tangent line to $\K$.
\end{itemize}
Then $(\P,\L,\I)$ is a GQ of order $(q-1,q+1)$, isomorphic to $AS(q)$ if $q$ is odd.
\end{theorem}

\begin{proof}
A line $L$ of $\PG(3,q)$, not in $H$, through a point $P$ of $\K$ contains $q$ affine points, and, since $P$ lies on exactly two tangent lines to $\K$, $L$ lies in exactly two planes through a tangent line to $\K$. This implies that a point of $\mathcal{P}$ lies on $q+2$ lines of $\L$.

An affine point of $\PG(3,q)$ lies on $q$ lines through a point of $\K$, and there are exactly $q$ affine lines through a point of $\K$ in a tangent plane to $\K$, which shows that a line of $\L$ contains $q$ points of $\P$.

Let $P$ be a point of $\P$ and let $L$ be a line of type (1) of $\L$, not through $P$. If the plane $\langle P,L\rangle$ is a secant plane to $\K$ then the unique line through $L$ and the second intersection point of $\langle P,L\rangle$ with $\K$, is the unique point of $\P$ on $L$ collinear with $P$. If the plane $\langle P,L\rangle$ meets the plane at infinity in a tangent line at the point $Q:=P\cap \K$, then the line $\langle L,Q\rangle$ is the unique point of $\P$ collinear with $P$ on $L$.

Let $P$ be a point of $\P$ and let $L$ be a line of type (2) of $\P$, not through $P$, say $L$ is tangent to $\K$ in the point $Q$. If $Q$ does not lie on $P$, then the plane $L$ and the line $P$ meet in a unique point $R$ which is affine. So the line $\langle R,Q\rangle$ is the unique point of $\P$ on $L$ collinear with $P$. If $Q$ lies on $P$, then let $T$ be the tangent line to $\K$ through $Q$, different from $L\cap H$. The intersection line of $\langle P,T\rangle$ with the plane $L$ is the unique point of $\P$ on $L$, collinear with $P$.

This proves that $(\P,\L,\I)$ is a GQ of order $(q-1,q+1)$.

Now, suppose $q$ is odd. By Theorem \ref{incompleteq}, the $q$-arc $\K$ lies on a unique conic $C$. Let $x$ be the point added to $\K$ to obtain $C$.

The obtained GQ is the Payne derivation of $T_2(C)^D$ at the point $x$: the points of the Payne derivation are the points, not collinear with $x$, hence these are all lines of $\PG(3,q)$, through $C\setminus\{x\}$. The first set of lines of the Payne derivation are the lines, not through $x$, hence, these are all affine points and all tangent planes to $C$, but not the tangent planes to $x$. The second set of lines of the Payne derivation are the hyperbolic lines through $x$. In this case, the set $\{x,y\}^{\perp \perp}$, where $y \not \sim x$ (hence, $y$ is a line through a point $Q \neq x$ of $C$) consists of the set of all affine lines in the plane $\langle y,x\rangle$. This set of lines can be identified with the plane $\langle y,x\rangle$, hence, the second set of lines are exactly the planes through $x$ and a point of $C$.

It was shown by Payne \cite{Payne3} that AS$(q)$ is obtained by Payne deriving W$(q)$ at a regular point. Now $T_2(C)$, where $C$ is a conic in $\PG(2,q)$, $q$ odd, is isomorphic to Q$(4,q)$ (see e.g. \cite[Theorem 3.2.2]{gq}), and Q$(4,q)$ is isomorphic to W$(q)^D$ (see e.g. \cite[Theorem 3.2.1]{gq}). This shows that AS$(q)$ and $(\P,\L,I)$ are Payne derived from W$(q)$ at a regular point, and hence, are isomorphic.
\end{proof}

\section{Pseudo-ovals and related objects}

A Laguerre plane is an incidence structure with points $\P$, lines $\L$ and circles $\C$ such that $(\P,\L,\C,\I)$ satisfies the following four axioms:

\begin{itemize}
\item[AX1] Every point lies on a unique line
\item[AX2] A circle and a line meet in a unique point
\item[AX3] Through $3$ points, no two on a line, there is a unique circle of $\C$
\item[AX4] If $P$ is a point on a fixed circle $C$ and $Q$ a point, not on the line through $P$ and not on the circle $C$, then there is a unique circle $C'$ through $P$ and $Q$, meeting $C$ only in the point $P$.
\end{itemize}

In a finite Laguerre plane, every circle contains $n+1$ points for some $n$; this constant $n$ is called the {\em order} of the Laguerre plane.

Starting from a point $P$ of a Laguerre plane $(\P,\L,\C,\I)$, we obtain an affine plane $(\P',\L',\I)$, where incidence is natural, as follows. 

\begin{itemize}
\item[$\P'$:] Points of $\P$, different from $P$ and not collinear with $P$
\item[$\L'$:] (1) lines of $\L$, not through $P$\\
(2) elements of $\C$, through $P$.
\end{itemize}
The obtained affine plane $(\P',\L')$ is called the {\em derived affine plane} at $P$.

The classical Laguerre plane $L(q)$ of order $q$ is obtained from a quadratic cone $K$ of $\PG(3,q)$, with vertex $V$. Points of $L(q)$ are points of $K$, different from $V$, lines of $L(q)$ are generators of $K$ and circles are the intersections of the planes, not through $V$, with $K$. 

A Laguerre plane is called {\em Miquelian} if for each eight pairwise different points $A,B,C,D,E, F,G,H$ it follows from $(ABCD)$,
$(ABEF)$, $(BCFG)$, $(CDGH)$, $(ADEH)$ that $(EFGH)$, where $(PQRS)$ denotes that $P,Q,R,S$ are on a common circle.

\begin{theorem} \cite{Chen}
If for one point of a Laguerre plane of odd order, the derived affine plane is Desarguesian, then the Laguerre plane is Miquelian.
\end{theorem}
By a theorem of van der Waerden and Smid \cite{vsw}, a Laguerre plane is Miquelian if and only if it is classical.

 A dual pseudo-oval is a set of $q^n+1$ $(2n-1)$-spaces in $\PG(3n-1,q)$ such that any three have an empty intersection. A dual pseudo-oval $\O$  in $\PG(3n-1,q)$ gives rise to a Laguerre plane $L(\O)$ of order $q^n$ in the following way. Embed $H=\PG(3n-1,q)$ as a hyperplane at infinity of $\PG(3n,q)$ and define $L(\mathcal{O})$ to be the incidence structure $(\P,\L,\C)$ with natural incidence and with
\begin{itemize}
\item[$\P$:] $2n$-spaces meeting $H$ in an element of $\O$
\item[$\L$:] elements of $\O$
\item[$\C$:] points of $\PG(3n,q)$, not in $H$.
\end{itemize}
In \cite{Steinke}, Steinke showed that every dual pseudo-oval gives rise to an {\em elation} Laguerre plane and vice versa (see also \cite{Lowen}). 
\begin{theorem} A Laguerre plane L is an elation Laguerre plane if and only if $L\cong L(\O)$ for some dual pseudo-oval $\O$.

\end{theorem}

In \cite{Thas}, Thas extends the construction of $T_2(O)$, where $O$ is an oval of $\PG(2,q)$ to a construction of a GQ $T(\O)$ of order $q^n$ from a pseudo-oval $\O$. The obtained GQ is a translation generalised quadrangle (TGQ).

Casse, Thas and Wild proved the following.
\begin{theorem}Ê\cite{Casse} \label{casse} If one affine plane arising from a (dual) pseudo-oval $\O$ in $\PG(3n-1,q)$, $q$ odd, is Desarguesian, then $T(\mathcal{O})$ is isomorphic to $T_2(O)$, where $O$ is an oval of $\PG(2,q^n)$, i.e., $\O$ is a pseudo-conic.
\end{theorem}

In the same way as Thas extended the construction for $T_2(O)$ to $T(\O)$, we can extend the construction of Theorem \ref{triviaal} to pseudo-arcs of size $q^n$ in order to obtain a GQ$(q^n-1,q^n+1)$. 

Let $\K$ be a pseudo-arc of size $q^n$ in $\PG(3n-1,q)$. Project the elements of $\K$ from an element $K_1$ of $\K$ onto a $(2n-1)$-space $\pi$, disjoint from $K_1$. The $q^n-1$ $(n-1)$-spaces in $\pi$ obtained in this way form a partial spread, which can be extended uniquely to a spread of $\PG(3n-1,q)$ by adding two $(n-1)$-spaces $\mu_1$, $\mu_2$ (see \cite{Beutelspacher}). This implies that the element $K_1$ lies on $2$ unique tangent $(2n-1)$-spaces to $\K$. It also shows that an $n$-space through an element $K_1$ of $\K$ either lies in a unique tangent space or meets exactly one element of $\K$, different from $K_1$ (necessarily in exactly one point).

\begin{theorem}\label{triviaal2} Let $\K$ be a pseudo-arc of size $q^n$ in $\PG(3n-1,q)$. Embed $H=\PG(3n-1,q)$ in $\PG(3n,q)$.
Let $(\P,\L,\I)$, where incidence is natural, be the following incidence structure.

\begin{itemize}
\item[$\P$:]  the $n$-spaces, not in $H$, through an element of $\K$.
\item[$\L$:] (1) the affine points of $\PG(3n,q)$\\
(2) the $2n$-spaces, not in $H$, through a tangent $(2n-1)$-space in $H$ to $\K$.
\end{itemize}
Then $(\P,\L,\I)$ is a GQ of order $(q^n-1,q^n+1)$. 
\end{theorem}
\begin{proof} It is clear that every line contains $q^n$ points and every point lies on $q^n+2$ lines. Let $P$ be a point of $\P$ and let $L$ be a line of $\L$ of the first type, not through $P$. Then $P$ corresponds to an $n$-space through an element $K_1$ of $\K$ and $L$ corresponds to an affine point, not on $P$. The $(n+1)$-space $\langle P,L\rangle$ meets $H$ in an $n$-space $\mu$ through $K_1$. If $\mu$ is contained in a tangent space $\rho$ (which is necessarily unique), then $\langle L,\rho\rangle$ is the unique line through $P$ intersecting $L$. If $\mu$ is not contained in a tangent space, it meets a unique element of $\K$, say $K_2$ in a point $Q$. The point $\langle L,K_2\rangle$ is the unique point on $L$ collinear with $P$. The points $P$ and $\langle L,K_2\rangle$ are collinear because these are two $n$-spaces, contained in the $2n$-space $\langle L,K_1,K_2\rangle$, hence they meet in an affine point.

Let $P$ be a point of $\P$ and let $L$ be a line of $\L$ of the second type, not through $P$. Then $P$ corresponds to an $n$-space through an element $K_1$ of $\K$ and $L$ corresponds to a $2n$-space through a tangent space $\mu$ at an element $K_2$ of $\K$. 

Suppose first that $K_1\neq K_2$. The $n$-space $P$ and $2n$-space $L$ in $\PG(3n,q)$ meet in a unique point since an $n$-space and a $2n$-space always meet, and if they would have a line in common, $K_1$ would meet the tangent space through $K_2$, which forces $K_1=K_2$, a contradiction. If $K_1=K_2$, then let $\mu_2$ be the second tangent space to $\K$ at $K_1=K_2$ and let $M$ be the $2n$-space spanned by $P$ and $\mu_2$, then $M$ and $L$ meet in an $n$-space through $K_1=K_2$, which is the unique point of $\P$ collinear with $P$ and lying on $L$. 
\end{proof}

\begin{remark} As in Theorem \ref{triviaal}, the Payne derivation of the quadrangle $T(\O)$, where $\O$ is a pseudo-oval, at an element $K_1$ of $\O$ can be seen as arising from the construction of Theorem \ref{triviaal2}, when we start with the pseudo-arc $\O\setminus{K_1}$.
\end{remark}

In the light of the previous remark, it makes sense to investigate whether complete pseudo-arcs of size $q^n$ in $\PG(3n-1,q)$ exist, where, as for arcs, a pseudo-arc is {\em complete} if it is not contained in a larger pseudo-arc. If so, new GQ$(q^n-1,q^n+1)$ arise. If not, all GQ obtained from pseudo-arcs of size $q^n$ in $\PG(3n-1,q)$ as in Theorem \ref{triviaal2} are Payne derivations of GQ obtained from pseudo-ovals.

\section{Orthogonal arrays}
An {\em orthogonal array} of {\em strength t}, on {\em s levels}, {\em k constraints}, and {\em index $\lambda$} is a $k \times N$ array (where $N=\lambda s^t$) with entries from a set of $s\geq 2$ symbols, having the property that in every $t \times N$ submatrix, every $t\times 1$ column vector appears exactly $\lambda$ times.

Labelling the lines of a Laguerre plane from $1$ to $n+1$, the points on each line from $1$ to $n$, and identifying each circle with the $(n+1)$-tuple $(c_1,\ldots, c_{n+1})$ where $c_i$ is the unique point of the circle on line $i$, we see that a Laguerre plane of order $n$ corresponds to an orthogonal array of strength $3$ on $n$ levels, $n+1$ constraints and index $1$ (and conversely).

A {\em Laguerre near-plane} of order $n\geq 3$ is an incidence structure of $n^2$ points, $n$ lines and $n^3$ circles satisfying the axioms (AX1), (AX2), (AX3) from above. A Laguerre near-plane of order $n$ corresponds to an orthogonal array of strength $3$ on $n$ symbols, $n$ constraints and index $1$.

In \cite{Steinke2}, the following theorem is shown.
\begin{theorem}\label{unique}
A finite Laguerre near-plane of odd order $n\geq 7$ that admits a point whose internal incidence structure extends to a Desarguesian plane can be uniquely extended to the Miquelian Laguerre plane of order $n$ by adjoining the points of one line.
\end{theorem}
This theorem also shows that the orthogonal array of strength $3$ on $n$ symbols, $n$ constraints and index $1$ that arises from a Laguerre near-plane admitting a point whose internal incidence structure extends to a Desarguesian plane, is uniquely extendable to an orthogonal array of strength $3$ on $n$ symbols, $n+1$ constraints and index $1$.

\section{Results for small orders}

From Theorem \ref{triviaal2}, we know that, starting from a pseudo-arc of size $4$ in $\PG(5,2)$, we obtain a GQ of order $(3,5)$. But any GQ of order $(3,5)$ is isomorphic to the $T_2^\ast(O)$ arising from a hyperoval in $\PG(2, 4)$ \cite{Dixmier}. This implies that a pseudo-arc of size $4$ in $\PG(5,2)$ lies in a pseudo-oval.

With the use of MAGMA \cite{Magma}, we were able to show that a pseudo-arc of size $8$ in $\PG(8,2)$ lies in a pseudo-oval. Probably, the methods developed in \cite{Steinke3} to show that finite Laguerre planes of order $8$ are ovoidal can be used to show this result by hand.

Furthermore, by computer, we showed that a pseudo-arc of size $9$ of $\PG(5,3)$ lies in a pseudo-oval.

\section{The proof of the main theorem}

\begin{lemma} \label{essential} Let $s>m_2'(q^n)$. Let $S=\{A_1,\ldots,A_{q^{3n}-q^{2n}}\}$ be a set of $q^{3n}-q^{2n}$ $s$-arcs through $(1,0,0)$ in $\PG(2,q^n)$, such that the $q^n-s+2$ tangent lines at the point $(1,0,0)$ to the $s$-arcs in $S$ coincide and such that two $s$-arcs of $S$ meet in at most $2$ points, different from $(1,0,0)$.

Let $\bar{S}$ be the set $\{C_1,\ldots,C_{q^{3n}-q^{2n}}\}$, where $C_i$ is the unique conic containing $A_i$.  Then all conics $C_i$ are tangent to the same line at $(1,0,0)$.
\end{lemma}
\begin{proof}
Let the tangent lines be $\ell_0$, with equation $Z=0$, and $\ell_i$ with equation $Y+\rho_iZ=0$, for $i=1,\ldots,q^n-s+1$, where $\rho_1=0$. 

Let $m_1,m_2,m_3$ be three distinct lines through $(1,0,0)$ that are not tangent lines to the $s$-arcs. Since every $s$-arc meets the lines $m_1$, $m_2$, $m_3$ in a point and these points are non-collinear, we have $q^{3n}-q^{2n}$ choices for these intersection points. Since there are exactly $q^{3n}-q^{2n}$ $s$-arcs in the set $S$ and two $s$-arcs meet in at most $2$ points, different from $(1,0,0)$, we get that there is exactly one $s$-arc of $S$ through three non-collinear points, not lying on the lines $\ell_0,\ldots,\ell_{q^n-s+1}$, that are such that no two of them are collinear with $(1,0,0)$. In the same way, we also get that there are $q^n-1$ $s$-arcs of $S$ through two points that do not lie on the lines $\ell_0,\ldots,\ell_{q^n-s+1}$ and are not collinear with $(1,0,0)$.

Every conic tangent to $\ell_0$ at $(1,0,0)$ has equation $XZ+dY^2+eYZ+fZ^2=0$, for some $d,e,f$ in $\F_{q^n}$, and $d\neq 0$. Every conic tangent to $\ell_i$ at $(1,0,0)$ has equation $XY+\rho_iXZ+d_iY^2+e_iYZ+f_iZ^2=0$, where $d_i,e_i,f_i$ are in $\F_{q^n}$, but $f_i\neq e_i\rho_i-d_i\rho_i^2$.

Now, we may assume without loss of generality that there is a conic $C$ of $\bar{S}$ tangent to $\ell_0$, hence, $C$ has equation $XZ+\bar{d}Y^2+\bar{e}YZ+\bar{f}Z^2=0$ for some fixed $\bar{d},\bar{e},\bar{f}$, where $\bar{d}\neq 0$. Every point on $C$, different from $(1,0,0)$, has coordinates $(-\bar{d}y^2-\bar{e}y-\bar{f},y,1)$.

Now let $P_0$ and $P_1$ be two points on $A$, the $s$-arc of $S$ contained in $C$. There are $q^n-1$ conics through $P_0,P_1$ in the set $\bar{S}$. Suppose a conic $C'$ through $P_0$ and $P_1$ in the set $\bar{S}$ is tangent to the line $\ell_i$ for some $i\geq 1$. The intersection points of $C$ and $C'$, different from $(1,0,0)$, are given by the points 

$(-\bar{d}y^2-\bar{e}y-\bar{f},y,1)$, where $y$ is a solution of
$$-\bar{d}y^3-\bar{e}y^2-\bar{f}y-\rho_i\bar{d}y^2-\rho_i\bar{e}y-\rho_i \bar{f}+d_i y^2+e_iy+f_i=0 \ (*),$$ where $d_i,e_i$ and $f_i$ are determined by the conic $C'$. Now $P_0=(-\bar{d}y_0^2-\bar{e}y_0-\bar{f},y_0,1)$ for some $y_0$ and $P_1=(-\bar{d}y_1^2-\bar{e}y_1-\bar{f},y_1,1)$ for some $y_1\neq y_0$. Note that $y_0y_1\neq 0$ since $Y=0$ is a tangent line to the $s$-arc $A$. Since the equation $(*)$ has two (different) solutions $y_0,y_1$ in $\F_{q^n}$, it has a third solution $y_2$ in $\F_{q^n}$, given by $y_2=(-\rho_i\bar{f}+f_i)/\bar{d}y_0 y_1$. If $y_2$ is not equal to a value in $\{y_0,y_1,-\rho_1,-\rho_2,\ldots,-\rho_{q^n-s+1}\}$, then we obtain a contradiction since $A$ and $A'$, where $A'$ is the $s$-arc of $S$ contained in $C'$ can only have two points in common. Expressing that $y_2$ equals one of the values in $\{y_0,y_1,-\rho_1,-\rho_2,\ldots,-\rho_{q^n-s+1}\}$ gives an equation determining $f_i$, which uniquely determines $d_i$ and $e_i$. Hence, there are at most $q^n-s+3$ conics through $P_0$ and $P_1$ in the set $\bar{S}$, that are not tangent to $\ell_0$, hence, at least $s-4$ conics through $P_0$ and $P_1$ in $\bar{S}$ are tangent to $\ell_0$. If there is a conic in $\bar{S}$ through $P_0$ and $P_1$ tangent to $\ell_i$, for some fixed $i\geq 1$, then the same argument yields that there are at least $s-3$ conics through $P_0$ and $P_1$ tangent to $\ell_i$. But then there would be $2s-8>q^n-1$ conics in $S$ through $P_0$ and $P_1$, a contradiction since $s>m_2'(q^n)$ implies that $s>(q^n+7)/2$. This implies that all conics of $\bar{S}$ meeting a conic of $\bar{S}$ tangent to $\ell_0$ in $2$ points, different from $(1,0,0)$, are tangent to $\ell_0$ too. Now let $P_2,P_3,P_4$ be points, not on one of the lines $\ell_j$, $j=0,\ldots,q^n-s+1$, and such that $P_2,P_3,P_4$ are not collinear and no $2$ points of $\{P_0,P_1,P_2,P_3,P_4\}$ are collinear with $(1,0,0)$. Then there is a unique $s$-arc through $P_0,P_1,P_2$ that is by the previous extending to a conic tangent to $\ell_0$. Repeating the same argument for the unique $s$-arcs through $P_1,P_2,P_3$ and $P_2,P_3,P_4$ resp., we obtain that all $s$-arcs extend to conics tangent to $\ell_0$.
\end{proof}

Let us fix some notation. Let $\K=\{K_1,\ldots,K_{s}\}$ be a dual pseudo-arc in $\PG(3n-1,q)$ with $s$ elements, where $s>m'_2(q^n)$. Let $\S_i:=\{K_j\cap K_i\vert j\neq i\}$. It is clear that for a fixed $i$, the set $\S_i$ is a partial $(n-1)$-spread of $K_i$.

Embed $H:=\PG(3n-1,q)$ as a hyperplane in $\PG(3n,q)$. Now define the following incidence structure $\G=(\P,\L,\C)$, where incidence is natural:
\begin{itemize}
\item [$\P$:] the $2n$-spaces through an element of $\K$, not contained in $H$
\item[$\L$:]  the elements of $\K$
\item[$\C$:] the points of $\PG(3n,q)$, not in $H$.
\end{itemize}
With these definitions, $\G$ has $sq^n$ points, $s$ lines and $q^{3n}$ circles where a line contains $q^n$ points, a circle contains $s$ points, and there are $q^{2n}$ circles through one point.

The incidence structure $\G$ satisfies the first $3$ axioms for a Laguerre plane:

\begin{itemize}
\item[AX1] {\em Every point of $\P$ lies on a unique line of $\L$}. This holds since a $2n$-space through an element of $\K$ contains a unique element of $\K$.

\item[AX2] {\em A circle of $\C$ and a line of $\L$ meet in a unique point of $\P$}. This holds since a point of $\PG(3n,q)$, not in $H$ and an element of $\K$ span a unique $2n$-space.

\item[AX3]  {\em Through $3$ points of $\P$, no two on a line of $\L$, there is a unique circle of $\C$}. This holds since three $2n$-spaces in $\PG(3n,q)$ meet in at least one point. If there would be a line contained in the intersection, there would be a point contained in the intersection of $3$ elements of $\K$, a contradiction.
\end{itemize}
Hence, $\G$ is an orthogonal array of strength $3$ on $q^n$ levels, $s$ constraints and index $1$.

\begin{theorem} \label{main1} Let $\G$ be an orthogonal array of strength $3$ on $q^n$ levels, with $s$ constraints and index $1$ obtained from $\K$, where $s>m_2'(q^n)$. If $\S_i$ extends to a Desarguesian spread for some $1\leq i\leq s$, then $\G$ can be uniquely extended to an orthogonal array of strength $3$ on $q^n$ levels with $n+1$ constraints and index $1$, which forms a Miquelian Laguerre plane.  \end{theorem}

\begin{proof}
Let $\S_1$ be the partial spread extending to a Desarguesian spread and fix a $2n$-space $P$ through $K_1$, not contained in $H$. We define a new incidence structure $(\P,\L,\I)$, where incidence is natural and: 
\begin{itemize}

\item[$\P$:]  $2n$-spaces through an element of $\K$, different from $K_1$
\item[$\L$:] (1)  elements of $\K \neq K_1$\\
 (2) the $q^{2n}$ points on $P$, not in $K_1$.
\end{itemize}
Now we claim that a circle of $\G$, not through the point $P$ of $\G$, becomes a set $A$ of $s-1$ points in $(\P,\L)$, such that every line of $\L$ has at most 2 points of $A$. A circle of $\G$, not through $P$, consists of a set $C$ of $s$ different $2n$-spaces through a fixed point $r$, not in $P$ and an element of $\K$. Since $\langle K_1,r\rangle$ is not a point of $(\P,\L)$, a circle corresponds to a set of $s-1$ points in $(\P,\L)$.

 A line of $\L$ of the first type contains a unique point of $A$. If a line of the second type would contain $3$ points of $A$, then the three $2n$-spaces would meet in a point of $P$ and a point, not in $P$, hence, in a line. This is a contradiction since it would imply that there are $3$ elements of $\K$ meeting in a point. This proves our claim.

By intersecting the points of $\P$ and the lines of $\L$ with the $2n$-space $P$, we obtain an incidence structure $(\P',\L')$ that is clearly isomorphic to $(\P,\L)$. To be more precise, the incidence structure $(\P',\L',\I)$ obtained (with natural incidence) has:

\begin{itemize}
\item[$\P':$]  $n$-spaces through an element of $\S_1$ in $P$
\item[$\L':$] (1)  elements of $\S_1$\\
(2) the $q^{2n}$ points on $P$, not in $K_1$.
\end{itemize}
It is clear that we still have the property that a circle of $\G$ corresponds to a set $A'$ of $s-1$ points in $(\P',\L')$ such that every line of $\L'$ has at most $2$ points of $A'$.

Since $\S_1$ is extendable to a Desarguesian spread $\D$, $(\P',\L')$ is embeddable in a Desarguesian projective plane $(\P'',\L'')$ with

\begin{itemize}
\item[$\P'':$] (1) $n$-spaces through the elements of $\D$ in $P$\\ (2) the space $K_1$ (also denoted by $(\infty)$)

\item[$\L'':$] (1) elements of $\D$\\ \ (2) the $q^{2n}$ points on $P$, not in $K_1$.
\end{itemize}
The set of $s-1$ points $A''$ obtained from the points on a circle of $\G$, not through the point $P$, has the property that no line of $(\P'',\L'')$ contains more than $2$ points of $A''$: this property held for all lines of type (2) in $(\P',\L')$, moreover, every line of type (1) that is an element of $\D$ but not of $\K$ contains zero points of $A''$.

The space $K_1$ is the union of all lines of type (1). This point can be added to $A''$ to obtain a set of $s$ points, no three on a line. The lines corresponding to the elements of $\D$, not in $\S_i$ are tangent to the set $A''$ in the point $(\infty)$.

Since $(\P'',\L'')$ is Desarguesian and $s>m_2'(q^n)$, we know that the $s$-arc $A''$ can be uniquely extended to a conic. This holds for all $q^{3n}-q^{2n}$ distinct $s$-arcs obtained from the points on the circles, not through $P$.

Hence, we have a set of $q^{3n}-q^{2n}$ distinct $s$-arcs in $\PG(2,q^n)$, such that the tangent lines at the point $(\infty)$ coincide. Moreover, by (AX3), two $s$-arcs meet in at most $2$ points, different from $(\infty)$. Let $S$ be the set of $q^{3n}-q^{2n}$ conics that are the extensions of the $s$-arcs obtained from the circles of $\G$, not through $P$. By Lemma \ref{essential}, the tangent line at the point $(\infty)$ coincides for every conic of $S$, denote it by $\ell_{\infty}$.

Define the following incidence structure.
\begin{itemize}
\item[$\P$:] points of $\PG(2,q^n)$, not on $\ell_\infty$,
\item[$\L$:] lines through $(\infty)$, different from $\ell_\infty$,
\item[$\C$:] (1) elements of $S$,\\
 (2) lines not through $(\infty)$.
 \end{itemize}
Then $(\P,\L,\C)$ is a Laguerre near-plane $\G'$ of order $q^n$. It clearly has $q^{2n}$ points, $q^n$ lines and $q^{3n}$ circles. Moreover, every point lies on a unique line, a circle and a line meet in a unique point of $\P$ and every $3$ non-collinear points determine a unique circle.
By Theorem \ref{unique}, this Laguerre near-plane can be extended uniquely to a Miquelian Laguerre plane $\mathcal{M}$.
\end{proof}

The results stated in the following lemma are well-known, but we give a proof for completeness. We denote the group of elations of $\PG(3n,q)$ with axis $H:X=0$ by $G_{el}$ and the group of perspectivities of $\PG(3n,q)$ with axis $H$ by $G_{per}$. Note that $|G_{el}|=q^{3n}$ and $|G_{per}|=q^{3n}(q-1)$.
\begin{lemma}  \label{1lemma} The following statements hold:
\begin{itemize}
\item $G_{el}$ can be identified with the additive group of a $3n$-dimensional vector space $V$ over $\F_q$.
\item $G_{el}$ is a normal subgroup of $G_{per}$.
\item Under this identification, a subgroup of $G_{el}$ forms a vector subspace of $V$ if and only if it is normalised by $G_{per}$.
\end{itemize}
\end{lemma}
\begin{proof} Let $\phi_v$, where $v$ is a vector of $\V(3n,q)$, denote the unique elation of $\PG(3n,q)$ with axis $X=0$ that maps the point $(1,0,\ldots,0)$ of $\PG(3n,q)$ to the point with coordinates $(1,v)$. Define $\phi_v+\phi_w$ to be the composition of the elations $\phi_{v}$ and $\phi_w$ and define $\lambda \phi_v$ to be $\phi_{\lambda v}$ where $\lambda\in \F_q$. It is clear that $\phi_v+\phi_w=\phi_{v+w}$ and it is not hard to check that with this addition and scalar multiplication, the set $\{\phi_v\vert v \in \V(3n,q)\}$ is a $3n$-dimensional $\F_q$-vector space.

Let $s$ be an element of $G_{per}$ and let $t$ be a non-trivial element of $G_{el}$, then $s^{-1}ts$ is a non-trivial element of $G_{per}$. Now suppose $s^{-1}ts$ fixes a point $P$, then $t$ fixes the point $s(P)$, and since $t$ is non-trivial, $s(P)$ is a point of $H$, hence, $s=s(P)$ belongs to $H$, which implies that $s^{-1}ts$ is an element of $G_{el}$.

The elation $\phi_v$, where $v=(x_1,\ldots,x_{3n})$, corresponds to the matrix obtained by replacing the first column of the identity matrix by the vector $(1,x_1,\ldots,x_{3n})$. Similarly, a perspectivity with axis $H$, say $\psi$, corresponds to the matrix obtained by replacing the first column of the identity matrix by the vector $(\lambda, y_1,\ldots,y_{3n})$ for some $\lambda\neq 0$ and $y_1,\ldots,y_{3n}$ not all zero in $\F_q$. It follows that $\psi^{-1}\phi_{v}\psi=\phi_{\lambda v}$. Hence, conjugating with an element of $G_{per}$ corresponds to scalar multiplication with an element of $\F_q$. This also implies that a subgroup is normalised under $G_{per}$ if and only if it is closed under scalar multiplication with elements of $\F_q$, so if and only if it forms an $\F_q$-vector subspace.
\end{proof}

\begin{lemma} The stabiliser $\bar{S}'$ of lines of a Miquelian Laguerre plane of order $q^n$ has order $q^{3n}(q^n-1)$.
\end{lemma}
\begin{proof} All automorphisms of the Miquelian Laguerre plane are induced by automorphisms of a quadratic cone with vertex $V$ in $\PG(3,q^n)$ (see \cite[Theorem 5.15]{Delandtsheer}). Hence, all elements in the stabiliser $\bar{S'}$ of the lines of $\M$ correspond to elements of $\PGammaL(4,q^n)$, fixing all lines of the quadratic cone, hence, fixing $V$ and all lines through $V$. The number of elements fixing all lines through a given point equals the number of collineations with a fixed axis in $\PG(3,q^n)$, which equals $q^{3n}(q^n-1)$. Moreover, it is clearly not possible for two different elements of the line-wise stabiliser of the point $V$ to induce the same automorphism of the quadratic cone.
\end{proof}

\begin{lemma} The stabiliser $\bar{S}$ of lines of $\G$ is a subgroup of $\bar{S}'$, the stabiliser of lines of a Miquelian Laguerre plane of order $q^n$.
\end{lemma}
\begin{proof}By Theorem \ref{main1}, the orthogonal array $\G$ of strength $3$ on $q^n$ symbols, $s$ constraints and index $1$, can be extended in a unique way to $\M$, where $\M$ is a Miquelian Laguerre plane. This implies that the automorphism group $Aut(\G)$ of $\G$ is a subgroup of the automorphism group $Aut(\mathcal{M})$ of $\mathcal{M}$.  This also implies that  $\bar{S}$, the stabiliser of the lines of $\G$ in $Aut(\G)$, is a subgroup of $\bar{S'}$, the stabiliser of lines of $\M$. 
\end{proof}

\begin{lemma} \label{3lemma} The stabiliser $\bar{S}$ of lines of $\G$ acts transitively on the circles of $\G$ and only the identity element fixes more than one circle (i.e. $\bar{S}$ is a Frobenius group). The number of fixed-point-free elements in $\bar{S}$ equals $q^{3n}-1$. The same holds for $\bar{S}'$.
\end{lemma}
\begin{proof} The group $G_{el}$ of elations with axis $H$ in $\PG(3n,q)$ clearly induces a subgroup of $\bar{S}$. As $G_{el}$ acts transitively on the points of $\PG(3n,q)$, not in $H$, the group $\bar{S}$ acts transitively on the circles of $\G$. An element $g$ which fixes two circles $C_1$ and $C_2$ of $\G$ induces an automorphism $\alpha$ of the derived structure at $P$, where $P$ is a point of $C_1$ or $C_2$, which completes uniquely to a projective plane $\pi$ of order $q^n$. 

Since $\alpha$ fixes $s-1$ lines through a point linewise, all intersection points of a fixed line with these $s-1$ lines are fixed pointwise and $\alpha$ fixes at least $s-3$ additional points, it easily follows that $\alpha$ induces the trivial automorphism $\beta$ of $\pi$ for every choice of $P$, so $g$ is trivial. 

Let $f$ be the number of elements of $\bar{S}$ fixing at least one circle. Count the number of couples $(C,g)$ where $C$ is a circle and $g$ is an element of $\bar{S}$ fixing $C$. Then, as $\bar{S}$ acts transitively on the circles and only the identity fixes more than 1 circle we get that this number equals $q^{3n}|\bar{S}_C|=f-1+q^{3n}$. By the orbit-stabiliser-theorem, $q^{3n}|\bar{S}_C|=|\bar{S}|$, which implies that $|\bar{S}|-f=q^{3n}-1$. The same argument works for $\bar{S}'$.
\end{proof}

\begin{lemma} \label{2lemma} Let $T$ be the set of fixed-point-free elements of $\bar{S}$ (the stabiliser of lines of $\G$), together with the identity. Then $T$ is a normal subgroup of $\bar{S}$ and $T$ can be identified with $G_{el}$. The same holds for the set of fixed-point-free elements $T'$ of $\bar{S}'$, together with the identity.
\end{lemma}
\begin{proof} Every element of $G_{el}$ induces an element of $T$ and $\vert G_{el}\vert=\vert T\vert$ by Lemmas \ref{1lemma} and \ref{3lemma}, so $T$ is a subgroup of $\bar{S}$. Let $s$ be an element of $\bar{S}$ and $t$ be an element of $T$, then it is clear that $s^{-1}ts$ cannot fix a circle, unless $t$ is the identity. This implies that $T$ is normal in $\bar{S}$.

Now $T$ is the unique subgroup of order $q^{3n}$ in $\bar{S}$ since every subgroup of order $q^{3n}$ of $\bar{S}$ is a Sylow $p$-subgroup, where $q$ is a power of the prime $p$, Sylow $p$-subgroups are conjugate, and $T$ is normal in $\bar{S}$. This implies that $T\cong G_{el}$.
 
 In the same way, $G_{el}$ induces a subgroup $T'$ of order $q^{3n}$ in $\bar{S'}$ which is unique in $\bar{S'}$ and hence, is isomorphic to $G_{el}$.
\end{proof}

\begin{lemma} \label{5lemma} The group $\bar{S}'$ has a unique subgroup $S$ of order $q^{3n}(q-1)$. This subgroup is contained in $\bar{S}$.
\end{lemma}
\begin{proof}
The group $\bar{S}$ is a subgroup of $\bar{S}'$, which has order $q^{3n}(q^n-1)$. Since $T$ is the unique subgroup of order $q^{3n}$ in $\bar{S}$ (by Lemma \ref{2lemma}), every subgroup of $\bar{S}$ of order $q^{3n}(q-1)$ contains $T$. The subgroups of $\bar{S}$ containing $T$ are in one-to-one correspondence with the subgroups of $\bar{S}/T$. We will show that $\bar{S}/T$ has a unique subgroup of order $q-1$. 

First, note that $\bar{S}/T\cong \bar{S}_C$, where $C$ is a circle: by the orbit-stabiliser theorem and the fact that $\bar{S}$ acts transitively on the circles, we have that $\vert \bar{S}_C\vert=\vert \bar{S}\vert/\vert T\vert$. Moreover, the mapping from $\bar{S_C}$ to $\bar{S}/T$ mapping $s$ to $sT$ is injective as $s$ fixes the points of $C$ and $T$ consists of fixed-point-free elements, so $\bar{S_C}$ is a subgroup of $\bar{S}/T$.

Let $P$ be a point in $C$ and consider the action of $S_C$ on the derived structure at $P$, which completes uniquely to a projective plane $\pi$, which is Desarguesian by the fact that $\G$ extends to a Miquelian Laguerre plane.

The group $S_C$ induces an automorphism group of $\pi$ fixing $s-1$ points on the line $L$ corresponding to $C$, and $s-1$ lines through the point $X$, which is the point at infinity of the lines of $\G$ linewise. It is not hard to show that such an automorphism has to be a homology with center $X$ and axis $L$. This group of homologies has order $q^n-1$ and is cyclic, hence, contains a unique subgroup of order $(q-1)$.

The same reasoning shows that $S$ is the unique subgroup of $\bar{S'}$ of order $q^{3n}(q-1)$. 
\end{proof}

\begin{theorem} \label{main2} Let $\K=\{K_1,\ldots,K_{s}\}$ be a dual pseudo-arc in $\PG(3n-1,q)$ with $s$ elements, where $s>m'_2(q^n)$. Let $\S_i:=\{K_j\cap K_i\vert j\neq i\}$. If $\S_i$ extends to a Desarguesian spread for some $1\leq i\leq s$, then $\K$ can be extended to a dual pseudo-conic.
\end{theorem}
\begin{proof} By the uniqueness of the group $T$ of size $q^{3n}$ and $S$ of size $q^{3n}(q-1)$ contained in $\bar{S}$ and $\bar{S}'$ obtained in Lemma \ref{2lemma} and Lemma \ref{5lemma}, we get that the action of $S$ on $T$ is the same for $\G$ and $\M$ and corresponds to the action of conjugating an elation of $\PG(3n,q)$ with axis $H$ with a homology of $\PG(3n,q)$ with axis $H$, i.e. with multiplication with an element of $\F_q$.

Let $\ell_1,\ldots,\ell_s$ be the lines of $\G$ and let $\ell_{s+1},\ldots,\ell_{q^n+1}$ be the lines of $\M$ that are not in $\G$. Let $P_i$ be a point on $\ell_i$. Consider the stabiliser $T_{P_i}$ of a point $P_i$ in the group $T$.

The action of $S$ on $T_{P_i}$ by conjugation normalises $T_{P_i}$: the point $P_i$ corresponds to a $2n$-space in $\PG(3n,q)$ through an element of $\K$. Let $s_i$ be a homology with centre $x_i$ and axis $H$, where $x_i$ lies in $P_i$. Then $s_i^{-1}T_{P_i}s_i=T_{P_i}$. Let $s_k$ be another element of $S$, then there is a (unique) element $t$ of $T$ that maps $s_i$ to $s_k$, so $s_k=ts_i$. It follows that $s_k^{-1}T_{P_i}s_k=s_i^{-1}t^{-1}T_{P_i}ts_i=s_iT_{P_i}s_i^{-1}=T_{P_i}$, because $t^{-1}T_{P_i}t=T_{P_i}$ as $T$ is abelian.  

By Lemma \ref{2lemma}, $T$ can be considered as a $3n$-dimensional vector space $V$ over $\F_q$. By Lemma \ref{1lemma}, we know that the fact that $T_{P_i}$ is normalised by $S$ implies that $T_{P_i}$ forms an $\F_q$-vector subspace of $V$. Since $| T_{P_i}|=q^{2n}$, these subspaces are $2n$-dimensional.

Since $T$ acts regularly on the circles of $\M$ and $3$ non-collinear points determine a unique circle, $T_{P_i}\cap T_{P_j}\cap T_{P_k}$, with $i,j,k$ mutually different, is the identity element. The space $\PG(V)$ can be identified with the hyperplane $H$, and it follows from the fact that $T_{P_i}\cap T_{P_j}\cap T_{P_k}$, with $i,j,k$ mutually different, is the identity element, that the $(2n-1)$-dimensional projective subspaces corresponding to $T_{P_i}$ form a dual pseudo-oval. Now $T_{P_i}$ corresponds to the element of $\K$ in $H$ corresponding to the line $\ell_i$. Hence, the set $T_{P_1},\ldots,T_{P_s}$ corresponds to the dual pseudo-arc $\K$ and $T_{P_1},\ldots,T_{P_{q^n+1}}$ corresponds to the dual pseudo-conic determining the Miquelian Laguerre plane $\M$. This implies that $\K$ can be extended to a dual pseudo-conic.
\end{proof}

\begin{corollary} \label{cor} Let $\K=\{K_1,\ldots,K_s\}$ be a pseudo-arc in $\PG(3n-1,q)$ of size $s> m'_2(q^n)$, $q$ odd. If for one element $K_i$ of $\K$, the partial spread $\S=\{K_1,\ldots,K_{i-1},K_{i+1},\ldots,K_{s}\}/K_i$ extends to a Desarguesian spread of $\PG(2n-1,q)=\PG(3n-1,q)/K_i$, then $\K$ is contained in a pseudo-conic.
\end{corollary}
\begin{proof} Denote the dual of a subspace $M$ of $\PG(3n-1,q)$ by $\bar{M}$. An element of $\S$, say $K_1/K_i$ equals $\langle K_1,K_i\rangle/K_i$. This space can be identified with $\overline{\langle K_1,K_i\rangle}$, which equals $\bar{K_1}\cap \bar{K_i}$. This implies that the set $\{K_1,\ldots,K_{i-1},K_{i+1},\ldots,K_{s}\}/K_i$  extends to a Desarguesian spread of $\PG(2n-1,q)$ if and only if $\{\bar{K_1}\cap \bar{K_s},\ldots,\bar{K_{i-1}}\cap \bar{K_s}, \bar{K_{i+1}}\cap \bar{K_s}, \ldots, \bar{K_1}\cap \bar{K_s}\}$ extends to a Desarguesian spread. Hence, the statement follows by dualising Theorem \ref{main2}.
\end{proof}
\begin{remark} If we apply Corollary \ref{cor} to a pseudo-oval, then we obtain Theorem \ref{casse}.\end{remark}


\begin{thebibliography}{99}
\bibitem{AS} R.W. Ahrens and G. Szekeres. On a combinatorial generalization of 27 lines associated with a cubic surface.
{\em J. Austral. Math. Soc.} {\bf 10} (1969), 485--492.

\bibitem{Barlotti} A. Barlotti. Un'osservazione intorno ad un teorema di B. Segre sui $q$-archi. {\em Matematiche (Catania)} {\bf 21} (1966), 23--29.  

\bibitem{Beutelspacher} A. Beutelspacher. Blocking sets and partial spreads in finite projective spaces. {\em Geom. Dedicata} {\bf 9 (4)} (1980), 425--449.

\bibitem{Magma} W. Bosma, J. Cannon, C. Playoust. The Magma algebra system. I. The user language. {\em J. Symbolic Comput.} {\bf 24 (3-4)}, (1997), 235Ð265.

\bibitem{Casse} L.R.A. Casse, J.A. Thas, and P.R. Wild. $(q^n+1)$-sets of ${\rm PG}(3n-1,q)$, generalized quadrangles and Laguerre planes.
{\em Simon Stevin} {\bf 59 (1)} (1985), 21--42.

\bibitem{Chen} Y. Chen and G. Kaerlein. Eine Bemerkung \"uber endliche Laguerre- und Minkowski-Ebenen.
{\em Geom. Dedicata} {\bf 2} (1973), 193--194.


\bibitem{Delandtsheer} A. Delandtsheer. {\em Dimensional linear spaces}. In F. Buekenhout, {\em Handbook of Incidence Geometries}. Buildings and Foundations, North-Holland, Amsterdam (1995), 193--295 (Chapter 6).



\bibitem{Dembowski} P. Dembowski. Finite geometries. Ergebnisse der Mathematik und ihrer Grenzgebiete, Band 44 Springer-Verlag, Berlin-New York, 1968.

\bibitem{Dixmier} S. Dixmier and F. Zara. Essai d'une m\'ethode d'\'etude de certains graphes li\'es aux groupes classiques. {\em C. R. Acad. Sci. Paris SŽr. A-B} {\bf 282 (5)} (1976), A259--A262. 


\bibitem{Hall} M. Hall, Jr. Affine generalized quadrilaterals. Studies in Pure Mathematics. {\em Academia Press, London} (1971), 113--116. 

\bibitem{Hirschfeld1} J.W.P. Hirschfeld and G. Korchm\'aros. On the embedding of an arc into a conic in a finite plane. {\em Finite Fields Appl.} {\bf 2} (1996), 274--292.
\bibitem{Hirschfeld2} J.W.P. Hirschfeld and G. Korchm\'aros. On the number of rational points on an algebraic
curve over a finite field. {\em Bull. Belg. Math. Soc. Simon Stevin} {\bf 5} (1998), 313--340.
\bibitem{Lowen} R. L\"owen. Topological pseudo-ovals, elation Laguerre planes, and elation generalized quadrangles.
{\em Math. Z.} {\bf 216 (3)} (1994), 347--369.

\bibitem{Men1} G. Menichetti. Sopra i $k$-archi completi nel piono grafico di traslazione di ordine $9$. {\em Matematiche (Catania)} {\bf 21} (1966), 150--156.

\bibitem{Men2} G. Menichetti. $q$-Archi completi nei piani di Hall di ordine $q=2^k$. {\em Atti Accad. Naz. Lincei Rend. Cl. Sci. Fis. Mat. Natur. (8)} {\bf 56} (1974), 518--525.

\bibitem{Payne0} S.E. Payne. Nonisomorphic generalized quadrangles. {\em J. Algebra} {\bf 18} (1971), 201--212.

\bibitem{Payne3} S.E. Payne. The equivalence of certain generalized quadrangles. {\em J. Combin. Theory} {\bf 10} (1971), 284--289.


\bibitem{gq} S.E. Payne and J.A. Thas. {\em Finite Generalized Quadrangles}. Pitman Advanced Publishing Program, 1984.

\bibitem{Payne2} S. Payne. Hyperovals and generalized quadrangles. Finite geometries (Winnipeg, Man., 1984), 251--270,
Lecture Notes in Pure and Appl. Math., 103, Dekker, New York, 1985.

\bibitem{Penttila} T. Penttila. Configurations of ovals. {\em J. Geom.} {\bf 76 (1--2)} (2003), 233--255. 

\bibitem{qvist} B. Qvist. Some remarks concerning curves of the second degree in a finite plane. {\em Ann. Acad. Sci. Fennicae. Ser. A. I. Math.-Phys.} {\bf 1952 (134)} (1952), 27 pp.

\bibitem{Segre2} B. Segre. Curve razionali normali e $k$-archi negli spazi finiti. {\em Ann. Mat. Pura Appl.} {\bf 39} (1955), 357--379.


\bibitem{Segre} B. Segre. Ovals in a finite projective plane. {\em Canad. J. Math.} {\bf 7} (1955), 414--416.




\bibitem{Steinke} G.F. Steinke. On the structure of finite elation Laguerre planes. {\em J. Geom.} {\bf 41 (1--2)} (1991), 162--179.
 

\bibitem{Steinke2} G.F. Steinke. Finite Laguerre near-planes of odd order admitting Desarguesian derivations. {\em
European J. Combin.} {\bf 21 (4)} (2000), 543--554.

\bibitem{Steinke3} G. Steinke. Finite Laguerre planes of order $8$ are ovoidal. {\em J. Combin. Theory, Ser. A} {\bf 102} (2003), 143--162.


\bibitem{Tallini} G. Tallini. Sui $q$-archi di un piano lineare finito di caratteristica $p=2$. {\em Atti Accad. Naz. Lincei Rend. Cl. Sci. Fis. Mat. Nat. (8)} {\bf 23} (1957), 242--245.

\bibitem{Thas2} J.A. Thas. The $m$-dimensional projective space $S_m(M_n(GF(q)))$ over the total matrix algebra $M_n(GF(q))$ of the $n\times n$-matrices with elements in the Galois field $GF(q)$. {\em Rend. Mat. (6)} {\bf 4} (1971), 459--532. 


\bibitem{Thas} J.A. Thas. Translation $4$-gonal configurations. {\em Atti Accad. Naz. Lincei Rend. Cl. Sci. Fis. Mat. Natur. (8)} {\bf 56 (3)} (1974), 303--314.



\bibitem{vsw} B.L. van der Waerden and L.J. Smid. Eine axiomatik der kreisgeometrie und der laguerregeometrie. {\em Math. Ann.} {\bf 110 (1)} (1935), 753--776.

\bibitem{Voloch} J.F. Voloch. Arcs in projective planes over prime fields. {\em J. Geom. } {\bf 38 (1--2)} (1990), 198--200.

\bibitem{Voloch2} J.F. Voloch. Complete arcs in Galois planes of nonsquare order. {\em Advances in finite geometries and designs} (Chelwood Gate, 1990), 401--406, Oxford Sci. Publ., Oxford Univ. Press, New York, 1991. 

\end{thebibliography}
\end{document}